\documentclass{amsart}
\usepackage[utf8]{inputenc}
\usepackage[english]{babel}
\usepackage{amsfonts}
\usepackage{commath}
\usepackage{amsthm}
\usepackage{amsmath}
\usepackage{amssymb}
\usepackage{tikz} 
\usepackage{etoolbox}
\usetikzlibrary{cd,external}
\usepackage{mathtools}
\usepackage[square,numbers]{natbib}
\usepackage[capitalise]{cleveref}

\theoremstyle{plain}
\newtheorem{theorem}{Theorem}[section]
\newtheorem{lemma}[theorem]{Lemma}
\newtheorem{proposition}[theorem]{Proposition}
\newtheorem{corollary}[theorem]{Corollary}
\newtheorem{question}[theorem]{Question}

\theoremstyle{definition}
\newtheorem{definition}[theorem]{Definition}
\newtheorem{example}[theorem]{Example}

\theoremstyle{remark}
\newtheorem{remark}[theorem]{Remark}

\newcommand\underrel[3][]{\mathrel{\mathop{#3}\limits_{%
      \ifx c#1\relax\mathclap{#2}\else#2\fi}}}

\newcommand{\N}{\mathbb{N}}
\newcommand{\Z}{\mathbb{Z}}
\DeclareMathOperator{\Aut}{Aut}
\DeclareMathOperator{\Hol}{Hol}
\newcommand{\pr}{\pi}
\DeclareMathOperator{\Ab}{Ab}

\newcommand{\id}{\mathrm{id}}
\DeclareMathOperator{\Triv}{Triv}
\DeclareMathOperator{\opTriv}{opTriv}
\newcommand{\op}{\mathrm{op}}
\begin{document}
\title{On Two-sided Skew Braces}

\author{S.~Trappeniers}

\address{Department of Mathematics and Data Science, 
 Vrije Universiteit Brussel, 
 Pleinlaan 2, 
 1050 Brussels, Belgium} 

\email{senne.trappeniers@vub.be} 

\subjclass[2020] {Primary 16T25, 20N99}

\keywords{Skew brace, Two-sided skew brace, Jacobson radical ring}

\begin{abstract}
    In order to study two-sided skew braces, we introduce the notion of weakly trivial skew braces. We give a classification of such skew braces and show that every two-sided skew brace is an extension of a weakly trivial skew brace by a two-sided brace. As a result, we obtain new and generalize known results relating the additive and multiplicative group of two-sided skew braces. Further, we show that two a priori different notions of prime and semiprime skew braces, as introduced by Konovalov, Smoktunowicz and Vendramin, coincide for two-sided skew braces.
\end{abstract}

\maketitle

\section{Introduction}
Left braces were introduced by Rump \cite{Rum07} as a generalization of Jacobson radical rings after he found that the latter give rise to solutions of the set-theoretic Yang--Baxter equation. This equation, proposed by Drinfel'd \cite{Dri92}, is a combinatorial version of the Yang--Baxter equation originating in the works of Yang and Baxter in statistical mechanics \cite{Bax72,Yan67}. Left braces were generalized to skew left braces by Guarnieri and Vendramin in \cite{GV17}. Since then, connections with exact factorizations of groups and Hopf--Galois extensions have been described \cite{SV18,ST22b}. As racks and quandles yield solutions of the set-theoretic Yang--Baxter equation, skew braces also have applications in knot theory \cite{Bac18}.

Jacobson radical rings are in bijective correspondence with two-sided braces and have been extensively studied; see for example \cite{ADS98,AS02,Sys12, Wat68}. The notion of two-sidedness equally makes sense for skew braces. So far, the only paper focusing on two-sided skew braces is due to Nasybullov \cite{Nas19}. Some of his ideas are pursued further in the present paper, in order to obtain that every two-sided skew brace is an extension of a weakly trivial skew brace by a two-sided brace. Weakly trivial skew braces are a new notion generalizing trivial and almost trivial skew braces and form the main topic of \cref{section: weakly trivial skew braces}. 
In \cref{section: twosided skew braces}, we then obtain our main results. For example, we show that aside from trivial and almost trivial skew braces on simple groups, the only simple two-sided skew braces are infinite two-sided braces. Also, we generalize some results of Jacobson radical rings to two-sided skew braces and improve some of the results obtained in \cite{Nas19} on the connection between the additive and multiplicative groups of two-sided skew braces. In \cite{KSV21} and its appendix, the authors introduced (strongly) prime and  (strongly) semiprime skew braces. For two-sided braces, both variations coincide with the usual notion for rings. It is then asked by the authors whether these variations always coincide. In \cref{section: prime and semiprime two-sided skew braces} we give an affirmative answer to this question in the case of two-sided skew braces.
\section{Preliminaries}
A \emph{skew (left) brace} is a triple $(A,+,\circ)$ where $A$ is a set and $+$ and $\circ$ are binary operations such that $(A,+)$ and $(A,\circ)$ are groups and the equality 
\begin{equation}
    a\circ(b+c)=a\circ b-a+a\circ c,\label{eq: left brace}
\end{equation}
holds for all $a,b,c\in A$. Here, $-a$ denotes the inverse of $a$ in the group $(A,+)$ and $\circ$ has precedence over $+$. The inverse of an element $a\in A$ with respect to the group $(A,\circ)$ is denoted by $\overline{a}$. The group $(A,+)$ is called the \emph{additive group} of A and $(A,\circ)$ is the \emph{multiplicative group} of $A$. It follows directly from \eqref{eq: left brace} that the identity elements of $(A,+)$ and $(A,\circ)$ coincide, this element is denoted by $0$. For skew braces $A$ and $B$, a map $f:A\to B$ is a \emph{skew brace homomorphism} if $f(a+b)=f(a)+f(b)$ and $f(a\circ b)=f(a)\circ f(b)$ for all $a,b\in A$. By $\Aut(A,+,\circ)$ we denote the group of \emph{skew brace automorphisms} of $A$, that is the bijective skew braces homomorphisms from $A$ to itself. A skew brace with an abelian additive group is called a \emph{brace}. If a skew brace $A$ has a nilpotent, respectively soluble, additive group, we say that $A$ is of \emph{nilpotent}, respectively \emph{soluble}, \emph{type}. A skew left brace is a \emph{two-sided skew brace} if for all $a,b,c\in A$ it holds that
\begin{equation*}
    (a+b)\circ c=a\circ c-c+b\circ c.
\end{equation*}
As an example, any group $G$ yields a two-sided skew brace by setting $g+h=g\circ h=gh$ for $g,h\in G$. We call this the \emph{trivial skew brace} on $G$ and denote it by $\Triv(G)$. 

For any $a\in A$, the map 
\begin{equation*}
    \lambda_a:A\to A,\quad b\mapsto -a+a\circ b,
\end{equation*}
is an automorphism of $(A,+)$ and moreover we obtain a group homomorphism 
\begin{equation*}
    \lambda:(A,\circ)\to \Aut(A,+),\quad a\mapsto \lambda_a,
\end{equation*}
which is called the \emph{$\lambda$-function} of $A$ \cite{GV17}.
A subset $I$ of $A$ is a \emph{left ideal} if it is a subgroup of $(A,+)$ and $\lambda_a(I)\subseteq I$ for all $a\in A$.
An \emph{ideal} $I$ of $A$ is a left ideal such that moreover $I$ is a normal subgroup of $(A,+)$ and $(A,\circ)$. For an ideal $I$, the \emph{quotient skew brace} $A/I$ can be defined in the natural way.

For any $a,b\in A$, we define $a*b=\lambda_a(b)-b=-a+a\circ b-b$. Intuitively, $*$ can be thought of as a measure of the difference between the operations $+$ and $\circ$. For $X,Y\subseteq A$, we define $X*Y$ as the additive subgroup generated by $\{x*y\mid x\in X, y\in Y\}$. For example, $A*A$ is an ideal of $A$ and it is the minimal one such that $A/(A*A)$ is a trivial skew brace; see \cite[Proposition 2.3]{CSV19}.

Following \cite{BCJO19,Rum07,Smo18}, we introduce four descending series of subskew braces. We set $A^{(1)}=A^1=A^{[1]}=A_1=A$. The right series of a skew brace $A$ is defined as $A^{(n+1)}=A^{(n)}*A$, where $n\geq 1$. $A$ is said to be \emph{right nilpotent} if $A^{(n)}=\{0\}$ for some $n$. Similarly, the left series of a skew brace $A$ is given by $A^{n+1}=A*A^n$, where $n\geq 1$. $A$ is said to be \emph{left nilpotent} if $A^{n }=\{0\}$ for some $n$. We can also inductively define $A^{[n+1]}$ as the additive subgroup generated by $\bigcup_{1\leq i\leq n}A^{[i]}*A^{[n+1-i]}$, where $n\geq 1$. We say that $A$ is \emph{strongly nilpotent} if $A^{[n]}=\{0\}$ for some $n$. At last, the derived series of a skew brace $A$ is defined inductively by $A_{n+1}=A_n*A_n$, where $n\geq 1$. A skew brace is \emph{soluble} if $A_n=\{0\}$ for some $n$. The relation between left, right and strongly nilpotency is given by the following result, which is \cite[Theorem 2.30]{CSV19}.
\begin{theorem}\label{theorem: left right strong nilpotency}
A skew left brace is strongly nilpotent if and only if it is both left and right nilpotent.
\end{theorem}

Given any skew brace $A$, we can replace the additive group by its opposite group and obtain a new skew brace $A_\op=(A,+_\op, \circ)$, called the \emph{opposite skew brace} of $A$ \cite{KT20}. The $\lambda$-map of $A_\op$ associated with an element $a$ is denoted $\lambda^\op_a$. Concretely, $\lambda^\op_a(b)=a\circ b-a=a+\lambda_a(b)-a$. As a consequence, ideals of $A$ and $A_\op$ coincide. The operation $*$ associated to $A_\op$ is denoted by $*_\op$, meaning that $a*_\op b=-b+a\circ b-a$. Note that $A*_\op A$, or shortly $A^2_\op$, is an ideal of $A_\op$ and therefore also of $A$. For a group $G$, the opposite skew brace of $\Triv(G)$ is called the \emph{almost trivial skew brace} on $G$ and we denote it by $\opTriv(G)$. Concretely $g\circ h=g+h=gh$ for $g,h\in \opTriv(G)$. Therefore $g*h=h^{-1}ghg^{-1}$ in $\opTriv(G)$ and thus $\opTriv(G)^2$ coincides with the derived subgroup of $G$. The ideals of $\opTriv(G)$ are easily seen to be the normal subgroups of $G$. Also, left, right and strong nilpotency coincide for $\opTriv(G)$ with nilpotency of the group $G$ in the classical sense. Similarly, solubility of the skew brace $\opTriv(G)$ and the group $G$ coincide.

Given skew braces $A$ and $B$, their \emph{direct product} $A\times B$ is a skew brace with underlying set $A\times B$ and for all $a,a'\in A$, $b,b'\in B$,
\begin{align*}
    (a,b)+(a',b')=(a+a',b+b');\quad
    (a,b)\circ(a',b')=(a\circ a',b\circ b').
\end{align*}

Recall that for any ring $R$ the operation $a\circ b=a+b+ab$ makes $(R,\circ)$ into a monoid. If $R$ is a group, then $R$ is called a \emph{Jacobson radical ring}. Equivalently, these are the rings that coincide with their own Jacobson radical. If $R$ is a Jacobson radical ring then $(R,+,\circ)$ is a two-sided brace and moreover every two-sided brace $(A,+,\circ)$ yields a Jacobson radical ring where the ring multiplication is given by $*$ \cite{Rum07}. This follows partially from the following straightforward lemma. 
\begin{lemma}\label{lem: distributivity skew}
Let $A$ be a skew left brace. Then for all $a,b,c\in A$, the following equalities hold:
\begin{align*}
    a*(b+c)&=a*b+b+a*c-b,\\
    (a\circ b)*c&=a*(b*c)+b*c+a*c.
\end{align*}
For a two-sided brace, also the following equalities hold:
\begin{align*}
    (a+b)*c&=-b+a*c+b+b*c,\\
    a*(b\circ c)&=a*c+b*c+(a*b)*c.
\end{align*}
\end{lemma}
Under the above correspondence, ideals of two-sided braces and Jacobson radical rings coincide. As ring multiplication is associative, the three different types of nilpotency coincide for two-sided braces and correspond with the classical notion for rings. The following result is proved for left braces in \cite[Remark 5.2]{CJO20} and one implication is proved for skew left braces in \cite[Lemma 4.1]{Nas19}. The proof of the other implication is the same as for left braces.
\begin{proposition}\label{prop: twosided iff mult conj additive automorphism}
A skew brace $A$ is two-sided if and only if all inner automorphisms of $(A,\circ)$ are skew brace automorphisms of $A$.
\end{proposition}
A useful consequence is the following result, found in the proof of \cite[Corollary 4.2]{Nas19}.
\begin{corollary}\label{cor: char is ideal}
Let $A$ be a two-sided skew brace and $I$ a characteristic subgroup of $(A,+)$, then $I$ is an ideal of $A$. 
\end{corollary}
At last, for a group $G$, its abelianization is denoted by $\Ab(G)$ and its center by $Z(G)$. The usual notation $[\cdot,\cdot]$ is used for the commutator.

\section{Weakly trivial skew braces} \label{section: weakly trivial skew braces}
In this section we introduce the new notion of weakly trivial skew braces. Other than being a generalization of both trivial and almost trivial skew braces, the real motivation for this definition will become clear in \cref{section: twosided skew braces}. 
\begin{definition}
A skew left brace $A$ is \emph{weakly trivial} if $A^2\cap A_\op^2=\{0\}$.
\end{definition}
\begin{example}
All trivial and almost trivial skew braces are weakly trivial.
\end{example}
\begin{definition}
Let $G$ and $H$ be groups, then a \emph{subdirect product} of $G$ and $H$ is a subgroup $F$ of $G\times H$ such that $\pi_G(F)=G$ and $\pi_H(F)=H$, where $\pi_G$, respectively $\pi_H$, is the canonical projection of $G\times H$ onto $G$, respectively $H$. A \emph{subdirect product of skew braces} is defined analogously.
\end{definition}
\begin{proposition}\label{prop: weakly is subdirect product}
A skew brace $A$ is weakly trivial if and only if it embeds into a product of a trivial and almost trivial skew brace. In particular, it is a subdirect product of $A/A^2$ and $A/A_\op^2$.
\end{proposition}
\begin{proof}
It is clear that subskew braces and direct products of weakly trivial skew braces are once again weakly trivial, hence one implication follows. For the converse implication, consider the projections $\pr_1:A\to A/A^2$ and $\pr_2:A\to A/A_\op ^2$. The skew brace homomorphism 
\[\iota:A\to A/A^2\times A/A_\op^2:a\mapsto (\pr_1(a),\pr_2(a)),\]
has kernel $A^2\cap A^2_\op=\{0\}$. As $\iota$ clearly is a subdirect product, $A/A^2$ is trivial and $A/A^2_\op$ is almost trivial, this concludes the proof.
\end{proof}
\begin{remark}
    A weakly trivial skew brace $A$ might be constructed as a subdirect product of a trivial and almost trivial skew brace in multiple ways. Take for example two non-abelian groups $G$ and $H$ and consider the direct, hence subdirect, product $A=\Triv(G)\times \opTriv(H)$. It is easily seen that $A/A^2\cong \Triv(G)\times \Triv(\Ab(H))$ and $A/A_\op^2\cong \Triv(\Ab(G))\times \opTriv(H)$, hence $A$ is a subdirect product of  $(\Triv(G)\times \Triv(\Ab(H)))\times (\Triv(\Ab(G))\times \opTriv(H))$. Although one might argue that it is more desirable to write $A$ as a direct product when possible, the embedding into $A/A^2\times A/A_\op^2$ provides a canonical choice which will prove to be useful in the classification in \cref{theorem: classification weakly trivial}.
\end{remark}
\begin{proposition}\label{prop: A/A^2 cap A^2op is weakly trivial}
Let $A$ be a skew brace, then $A/(A^2\cap A^2_\op)$ is a weakly trivial skew brace.
\end{proposition}
\begin{proof}
Note that we have a natural embedding $\iota:A/(A^2\cap A^2_\op)\to A/A^2\times A/A^2_\op$. The statement then follows from \cref{prop: weakly is subdirect product}.
\end{proof}
\begin{corollary}
    Every weakly trivial skew brace is two-sided.
\end{corollary}
\begin{proof}
Let $A$ be a weakly trivial skew brace. As $A/A^2$ and $A/A^2_\op$ are two-sided, so is $A/A^2\times A/A^2_\op$ and therefore also $A$ by \cref{prop: weakly is subdirect product}.
\end{proof}
We now classify all weakly trivial skew braces. Our main tool for this is a generalisation of Goursat's lemma \cite{Gou89} to skew braces. For this we first need the existence of pullbacks in the category of skew braces, whose proof is left to the reader.
\begin{proposition}
Let $B,C,D$ be skew braces with skew brace homomorphisms $f:B\to D$, $g:C\to D$. Then the pullback of \[
\begin{tikzcd}
                 & B \arrow[d, "f"] \\
C \arrow[r, "g"] & D               
\end{tikzcd}\]
exists and is, up to isomorphism, given by
\[B\times_D C:=\{(a,b)\mid f(a)=g(b)\}\subseteq B\times C,\]
together with the projection maps $\pr_B:B\times_D C\to B$ and $\pr_C:B\times_D C\to C$.
\end{proposition}

The following version of Goursat's lemma holds for skew braces.
\begin{lemma}\label{lem: goursat}
There is a bijective correspondence between subdirect products of skew braces $B$ and $C$ and triples $(I,J,\rho)$ where $I$, respectively $J$, is an ideal of $B$, respectively $C$, and $\rho:B/I\to C/J$ is an isomorphism of skew braces.
\end{lemma}
\begin{proof}
We only give a sketch of the construction. Further details are just as in the classical case and are left to the reader. Let $A$ be a subdirect product of $B$ and $C$, where we identify $A$ with the image of its embedding. Let $I\subseteq B$ be $B$ such that $I\times \{0\}=A\cap (B\times \{0\})$ and $J\subseteq C$ such that $\{0\}\times J=A\cap (\{0\}\times C)$. Then $I$, respectively $J$, is an ideal of $B$, respectively $C$, and the map $\rho:B/I\to C/J$, given by $\rho(a)=b$ if and only if $(a,b)\in A$, is a well-defined skew brace isomorphism. We therefore obtain a triple $(I,J,\rho)$. 

Conversely, if a triple $(I,J,\rho)$ is given, the pullback of
\[
\begin{tikzcd}
            &                       & B \arrow[d] \\
C \arrow[r] & C/J \arrow[r, "\rho"] & B/J        
\end{tikzcd}\]
yields a subdirect product of $B$ and $C$.
\end{proof}
\begin{lemma}\label{lem: A fits in pullback diagram}
Let $A$ be a weakly trivial skew brace, then $A$ fits in the following pullback diagram.
\[
\begin{tikzcd}
A \arrow[d] \arrow[r] & A/A^2 \arrow[d] \\
A/A^2_\op \arrow[r]   & A/(A^2+A^2_\op)  
\end{tikzcd}\]
\end{lemma}
\begin{proof}
We denote the image of the embedding of $A$ into $A/A^2\times A/A^2_\op$ by $A'$ and now proceed as explained in the proof of \cref{lem: goursat}. Clearly $$A'\cap (A/A^2\times \{0\})=\{(a+A^2,0)\mid a\in A^2_\op\}=(A^2)_\op^2\times \{0\}.$$
Likewise, one finds that $A'\cap (\{0\}\times A/A^2_\op)=\{0\}\times (A/A^2)^2_\op$. It now is easily verified that $(A/A^2_\op)/(A/A^2_\op)^2\cong A/(A^2+A^2_\op)$ and $(A/A^2)/(A/A^2)^2_\op\cong A/(A^2+A^2_\op)$. The isomorphism 
\begin{equation*}
    \rho: A/(A^2+A^2_\op)\cong (A/A^2_\op)/(A/A^2_\op)^2\to (A/A^2)/(A/A^2)^2_\op\cong A/(A^2+A^2_\op),
\end{equation*}
must clearly be the identity in order to make the above diagram commute.
\end{proof}
\begin{definition}\label{def: triples weakly trivial}
Consider the class of triples $(G,H,\theta)$ with $G$ and $H$ groups and $\theta: \Ab(G)\to \Ab(H)$ an isomorphism. We say that two such triples $(G_1,H_1,\theta_1)$ and $(G_2,H_2,\theta_2)$ are equivalent if there exist group isomorphisms $\phi_G: G_1\to G_2$, $\phi_H:H_1\to H_2$ such that 
\[
\begin{tikzcd}
G_1 \arrow[r] \arrow[d, "\phi_G"] & \Ab(G_1) \arrow[r, "\theta_1"] \arrow[d, "\overline{\phi_G}"] & \Ab(H_1) \arrow[d, "\overline{\phi_H}"] & H_1 \arrow[l] \arrow[d, "\phi_H"] \\
G_2 \arrow[r]                     & \Ab(G_2) \arrow[r, "\theta_2"]                                & \Ab(H_2)                                & H_2 \arrow[l]                    
\end{tikzcd}\]
commutes, where $\overline{\phi_G}$ and $\overline{\phi_H}$ are the unique induced group isomorphisms making the square on the left and right hand side commute. 
\end{definition}
\begin{theorem}\label{theorem: classification weakly trivial}
There exists a bijection between isomorphism classes of weakly trivial skew braces and equivalence classes of triples as described in \cref{def: triples weakly trivial}. 
\end{theorem}
\begin{proof}
Let $A$ be a weakly trivial skew brace. We know that $\Ab(A/A^2,\circ)\cong (A/(A^2+A^2_\op),\circ)$ and $\Ab(A/A^2_\op,\circ)\cong (A/(A^2+A^2_\op),\circ)$. Hence $((A/A^2,\circ),(A/A^2_\op,\circ),\rho)$, where $\rho$ is the canonical group isomorphism, is a well-defined triple.

Conversely, given a triple $(G,H,\rho)$ we can construct the pullback of the diagram
\[
\begin{tikzcd}
                     &                                   & \Triv(G) \arrow[d] \\
\opTriv(H) \arrow[r] & \opTriv(\Ab(H)) \arrow[r, "\rho"] & \Triv(\Ab(G))     
\end{tikzcd}\]
which clearly is a weakly trivial skew brace. 

It follows from \cref{lem: A fits in pullback diagram} that if we start from a weakly trivial skew brace, consider its associated triple $((A/A^2,\circ),(A/A^2_\op,\circ),\rho)$ and then again take the pullback as described above we end up with a skew brace isomorphic to $A$. 

Conversely, start from a triple $(G,H,\rho)$ and let $A\subseteq \Triv(G)\times \opTriv(H)$ be its associated pullback. We want to prove that the triple associated to $A$ is isomorphic to $(G,H,\rho)$, which we can do by showing that the kernels of the vertical maps in the following commutative diagram are $A^2,A^2+A^2_\op, A^2+A^2_\op$ and $A^2_\op$ respectively. 
\[
\begin{tikzcd}
(A,\circ) \arrow[r,"\id"] \arrow[d, "\pr_G"] & (A,\circ)  \arrow[r,"\id"] \arrow[d] & (A,\circ) \arrow[d] & (A,\circ) \arrow[l,"\id"] \arrow[d, "\pr_H"] \\
G \arrow[r]                     & \Ab(G) \arrow[r, "\theta_2"]                                & \Ab(H)                                & H \arrow[l]                    
\end{tikzcd}\]
The kernel of the first vertical map is clearly $A\cap (\{0\}\times H)=\{0\}\times [H,H]$. Also, $A^2$ is generated by the elements $(g_1,h_1)*(g_2,h_2)=(h_2^{-1}h_1h_2h_1^{-1})$ for $(g_1,h_1),(g_2,h_2)\in A$. As every element of $H$ can appear as the second coordinate of an element of $A$, the equality $A\cap (\{0\}\times H)=A^2$ follows. Similarly it follows that the kernel of $\pi_H$ is equal to $A^2_\op$. The kernel of the second vertical map must be the commutator subgroup of $(A/A^2,\circ)$, which is $(A/A^2)_\op^2$, and similar for the third vertical map.
\end{proof}
When one looks at small weakly trivial skew braces, using for example the GAP-package \cite{YBGAP}, they all seem to have an isomorphic additive and multiplicative group. The following example shows that this is not always the case.
\begin{example}
Consider the group $$G:=\langle a,b\mid a^5=b^4=0,b^{-1}ab=a^2\rangle\cong C_5\rtimes C_4.$$
Its derived subgroup is the subgroup generated by $a$, hence $\Ab(G)\cong C_4$ is generated by the equivalence class of $b$. Let $A$ be the weakly trivial skew brace associated to the triple $(G,G,\id)$. It is easily seen that 
\[A=\{(a^kb^l,a^mb^l)\mid k,l,m\in \Z\}\subseteq \Triv(G)\times \opTriv(G).\]
One easily verifies that $(A,\circ)$ is isomorphic to the semidirect product $$C_5^2\rtimes_1 C_4:=\langle x,y,z\mid x^5=y^5=z^4=0,xy=yx,z^{-1}xz=x^2,z^{-1}yz=y^2\rangle,$$
where $(a,0)\mapsto x$, $(0,a)\mapsto y$ and $(b,b)\mapsto z$. Meanwhile, as $-b+a+b=bab^{-1}=a^{-2}=a^3$ in $\opTriv(G)$, we find that $(A,+)$ is isomorphic to  
$$C_5^2\rtimes_2 C_4:=\langle x,y,z\mid x^5=y^5=z^4=0,xy=yx,z^{-1}xz=x^2,z^{-1}yz=y^3\rangle,$$
where $(a,0)\mapsto x$, $(0,a)\mapsto y$ and $(b,b)\mapsto z$.
However, $C_5^2\rtimes_1 C_4$ is not isomorphic to $C_5^2\rtimes_2 C_4$, as in the first group all subgroups of order 5 are normal, but in the latter the subgroup of order 5 generated by $xy$ is not normal. 
\end{example}
\begin{example}
The previous example can be generalized by replacing $G$ by $\Hol(C_p)=C_p\rtimes \Aut(C_p)$ with $p>3$ a prime. In this way, one obtains an infinite amount of weakly trivial skew braces with non-isomorphic additive and multiplicative group.
\end{example}
Although the class of weakly trivial skew braces is closed under taking direct products and subskew braces, it is only closed under quotients by an ideal if the ideal satisfies some extra property.
\begin{lemma}\label{lem: ideal of weakly trivial}
Let $A$ be weakly trivial skew brace and consider its canonical embedding $\iota: A\to A/A^2\times A/A^2_\op$. Then $I\subseteq A$ is an ideal of $A$ if and only if $\iota(I)$ is a normal subgroup of $(A/A^2\times A/A^2_\op,\circ)$.
\end{lemma}
\begin{proof}
Clearly $\lambda_{(a_1,a_2)}(b_1,b_2)=(b_1,a_2\circ b_2\circ \overline{a_2} )$ and $\lambda^\op_{(a_1,a_2)}(b_1,b_2)=(\overline{a_1}\circ b_1\circ a_1,b_2)$, for all $(a_1,a_2),(b_1,b_2)\in \iota(A)$. As the projections $\iota(A)\to A/A^2$ and $\iota(A)\to A/A^2_\op$ are surjective, the result follows.
\end{proof}

\begin{lemma}\label{lem: quotient of weakly trivial}
Let $A$ be a weakly trivial skew brace and $I$ an ideal of $A$, then $A/I$ is weakly trivial if and only if $(I\cap A^2)+(I\cap A^2_\op)=I\cap (A^2+A^2_\op)$. 
\end{lemma}
\begin{proof}
Let $I$ be an ideal of $A$ and consider the skew brace $A/I$. Assume that $a+I\in (A/I)^2\cap (A/I)^2_\op$. This means that there exist elements $b\in A^2$ and $c\in A^2_\op$ such that $a+I=b+I=c+I$. If $a\notin I$ then also, $b,c\notin I$. Hence, $b-c\in I\cap (A^2+A^2_\op)$ is an element contained in $I\cap (A^2+A^2_\op)$ but not in $(I\cap A^2)+(I\cap A^2_\op)$. Also, every element $b+c$ in $I\cap (A^2+A^2_\op)\setminus ((I\cap A^2)+(I\cap A^2_\op))$, where $b\in A^2$ and $c\in A^2_\op$, yields a non-trivial element $b+I=-c+I\in (A/I)^2\cap (A/I)^2_\op$.
\end{proof}
\begin{example}
Let $G=D_8$, the dihedral group of order 8. $A$ be the weakly trivial skew brace associated to the triple $(G,G,\id)$. Then,
\begin{equation*}
    A=\{(g,h)\mid g+[G,G]=h+[G,G]\},
\end{equation*}
where we set $G=D_8$. In particular, $I=\{(g,g)\mid g\in Z(G)\}\subseteq A$ as $Z(G)=[G,G]$. By \cref{lem: ideal of weakly trivial}, it follows that $I$ is an ideal of $A$. But $I\cap A^2=I\cap A^2_\op=\{0\}$, so from \cref{lem: quotient of weakly trivial} it follows that $A/I$ is not weakly trivial. The same argument can be repeated for any group $G$ such that $Z(G)\cap [G,G]$ is non-trivial.
\end{example}
To conclude this section, we study some structural properties of weakly trivial skew braces. The following lemma is left as an exercise.
\begin{lemma}\label{lem: derived length subdirect product}
Let $\iota: G\to G_1\times G_2$ be a subdirect product of groups. Then $G$ is soluble if and only if $G_1\times G_2$ is soluble and in that case the derived length of $G$ and $G_1\times G_2$ coincide. Similarly, $G$ is nilpotent if and only if $G_1\times G_2$ is nilpotent and the nilpotency class of $G$ and $G_1\times G_2$ coincide.
\end{lemma}

\begin{corollary}\label{cor: derived length weakly trivial}
Let $A$ be a weakly trivial skew brace, then $(A,+)$ is soluble if and only if $(A,\circ)$ is soluble. In that case, their derived lengths coincide and $A$ is soluble as a skew brace. 
\end{corollary}
\begin{proof}
Consider the subdirect product $\iota:A\to A/A^2\times A/A^2_\op$, which implies that $\iota:(A,+)\to (A/A^2,+)\times (A/A^2_\op,+)$ and $\iota:(A,\circ)\to (A/A^2,\circ)\times (A/A^2_\op,\circ)$ are subdirect products of groups. As $(A/A^2,+)\times (A/A^2_\op,+)\cong (A/A^2,\circ)\times (A/A^2_\op,\circ)$, the first part of the statement follows from \cref{lem: derived length subdirect product}. Now if $A$ has a soluble additive subgroup, so does $A/A_\op^2$. In particular, $A/A_\op^2$ is soluble. As $A/A^2$ is trivial, hence soluble, we conclude that $A/A^2\times A/A_\op^2$, and therefore also $A$, is soluble.
\end{proof}
The proof of \cref{cor: derived length weakly trivial} can easily be adapted to prove the following corollary.
\begin{corollary}\label{cor: nilpotency group weakly trivial}
Let $A$ be a weakly trivial skew brace, then $(A,+)$ is nilpotent if and only if $(A,\circ)$ is nilpotent. In that case, their nilpotency classes coincide and the skew brace $A$ is left and right nilpotent.
\end{corollary}
\begin{proposition}\label{prop: left and right nilpotency coincide weakly trivial}
Let $A$ be a weakly trivial skew brace, then the following are equivalent:
\begin{enumerate}
    \item $A$ is left nilpotent,
    \item $A$ is right nilpotent,
    \item $A$ is strongly nilpotent.
\end{enumerate}
\end{proposition}
\begin{proof}
For any skew brace, 3 holds whenever 1 and 2 hold by \cref{theorem: left right strong nilpotency}. Therefore, only the equivalence of 1 and 2 has to be proved.

Assume that $A$ is left nilpotent. Then the almost trivial skew brace $A/A_\op^2$ is left nilpotent, hence right nilpotent. As $A/A^2$ is trivial so right nilpotent, it follows that $A/A^2\times A/A^2_\op$, and therefore also its subskew brace $A$, is right nilpotent. The same argument can be used to show that right nilpotency implies left nilpotency.
\end{proof}

\section{Two-sided skew braces}\label{section: twosided skew braces}
We start this section by proving our main results, \cref{theorem: A two-sided additive group intersection A^2 Aop^2 abelian} and \cref{cor: extension weakly trivial by brace}. For this we need the following lemma.
\begin{lemma}\label{lem: twosided skew brace additive commutator A*A A*opA is zero}
Let $A$ be a two-sided skew brace then $A^2$ and $A_\op^2$ centralize one another in $(A,+)$.
\end{lemma}
\begin{proof}
Let $a,b,c,d\in A$. Using consequently the fact that $A$ is a left and right skew brace
\begin{align*}
    (a+b)\circ(c+d)&=(a+b)\circ c-(a+b)+(a+b)\circ d\\
    &=a\circ c-c+b\circ c-b-a+a\circ d-d+b\circ d\\
    &=a\circ c+b*_\op c+a*d+b\circ d.
\end{align*}
If we start by using the right skew brace structure followed by the left one then we find
\begin{align*}
    (a+b)\circ(c+d)&=a\circ (c+d)-(c+d)+b\circ (c+d)\\
    &=a\circ c-a+a\circ d-d-c+b\circ c-b+b\circ d\\
    &=a\circ c+a*d+b*_\op c+b\circ d.
\end{align*}
Comparing both calculations we find $ b*_\op c+a*d=a*d+b*_\op c$, from which the statement follows.
\end{proof}
\begin{remark}
Although this result is new, a similar calculation appeared in \cite[Lemma 4.5]{Nas19} to prove that $A*Z(A,\circ)+Z(A,\circ)*A$ has an abelian additive group.
\end{remark}

\begin{theorem}\label{theorem: A two-sided additive group intersection A^2 Aop^2 abelian}
Let $A$ be a two-sided skew brace, then $A^2\cap A_\op^2$ is contained in $Z(A^2+A_\op^2,+)$. In particular, $A^2\cap A^2_\op$ is a two-sided brace.
\end{theorem}
\begin{proof}
Using \cref{lem: twosided skew brace additive commutator A*A A*opA is zero} we find
that $A^2\cap A^2_\op$ is in the center of both $(A^2,+)$ and $(A^2_\op,+)$, so it is contained in the center of $(A^2+A^2_\op,+)$. In particular, the commutativity of $(A^2\cap A^2_\op,+)$ follows.
\end{proof}
\begin{corollary}\label{cor: extension weakly trivial by brace}
Every two-sided skew brace is the extension of a weakly trivial skew brace by a two-sided left brace.
\end{corollary}
\begin{proof}
It suffices to note that the ideal $A^2\cap A_\op^2$ is a two-sided brace by \cref{theorem: A two-sided additive group intersection A^2 Aop^2 abelian} and $A/(A^2\cap A_\op^2)$ is a weakly trivial skew brace by \cref{prop: A/A^2 cap A^2op is weakly trivial}.
\end{proof}
On the other hand, it is not generally true that any extension of a weakly trivial skew brace by a two-sided brace is a two-sided skew brace, as we now demonstrate.
\begin{example}
Let $A$ be the semidirect product $\Triv(C_2)\ltimes \Triv(C_3)$, where $\Triv(C_2)$ acts non-trivially; see \cite[Corollary 3.36]{SV18}. Then $A$ is clearly an extension of the weakly trivial skew brace $\Triv(C_2)$ by the two-sided brace $\Triv(C_3)$. However, $C_2\times \{0\}$ is a characteristic subgroup of $(A,+)\cong C_2\times C_3$ but not a normal subgroup of $(A,\circ)\cong C_2\ltimes C_3$. It follows that $A$ can not be two-sided, as this would contradict \cref{cor: char is ideal}.
\end{example}
\begin{definition}
A skew brace $A$ is \emph{simple} if the only ideals are $\{0\}$ and $A$.
\end{definition}

\begin{theorem}\label{theorem: simple two-sided skew braces}
Let $A$ be a simple two-sided skew brace, then one of the following holds
\begin{enumerate}
    \item $A\cong \Triv(G)$ for a simple group $G$,
    \item $A\cong \opTriv(G)$ for a simple group $G$,
    \item $A$ is a simple two-sided brace.
\end{enumerate}
\end{theorem}
\begin{proof}
Let $A$ be a simple two-sided skew brace. If $A^2=\{0\}$, this means that $A=\Triv(G)$ for some group $G$. As the ideals of $A$ are precisely normal subgroups of $G$, we conclude that $G$ is simple. If $A_\op^2=\{0\}$ then the same reasoning yields that $A\cong \opTriv(G)$ for some simple group $G$. The only case which remains is the one where $A^2=A_\op^2=A$. By \cref{lem: twosided skew brace additive commutator A*A A*opA is zero} this implies that $(A,+)$ is abelian. Hence, $A$ is a simple two-sided brace.  
\end{proof}
\begin{corollary}
Let $A$ be a finite simple two-sided skew brace, then either $A\cong \Triv(G)$ or $A\cong \opTriv(G)$ for some finite simple group $G$.
\end{corollary}
\begin{proof}
It is well-known that if $A$ is a finite two-sided brace, then $A$ is strongly nilpotent and thus $A^2\neq A$. If furthermore, $A$ is simple it quickly follows that $A\cong \Triv(G)$ for some abelian group $G$. If we combine this observation with \cref{theorem: simple two-sided skew braces}, the statement follows.
\end{proof}
\begin{remark}
In \cite[Corollary 4.2]{Nas19} it was already proved that every simple finite two-sided skew brace of soluble type is trivial.
\end{remark}
\begin{remark}
    Note that there exist simple non-trivial Jacobson radical rings, and therefore also simple non-trivial two-sided braces. The first such example was constructed in \cite{SC67}.
\end{remark}
\begin{theorem}\label{theorem: two-sided skew brace circ soluble then + soluble}
Let $A$ be a two-sided skew brace. If $(A,\circ)$ is soluble of derived length $n$, then $(A,+)$ is soluble of derived length at most $n+1$. 
\end{theorem}
\begin{proof}
By \cref{prop: A/A^2 cap A^2op is weakly trivial}, $A/(A^2\cap A_\op^2)$ is weakly trivial. Because $(A/(A^2\cap A_\op^2),\circ)$ has derived length at most $n$, it follows from \cref{cor: derived length weakly trivial} that $(A/(A^2\cap A_\op^2),+)$ has derived length at most $n$. We also know that $(A^2\cap A_\op^2,+)$ is abelian by \cref{theorem: A two-sided additive group intersection A^2 Aop^2 abelian}, from which we conclude that $(A,+)$ has derived length at most $n+1$. 
\end{proof}
\begin{remark}
    \cref{theorem: two-sided skew brace circ soluble then + soluble} is a refinement of {\cite[Theorem 4.6]{Nas19}}, where nilpotency instead of solubility of the multiplicative group was assumed and the upper bound was $2n$.
\end{remark}
\begin{lemma}\label{lem: circ nilpotent two-sided then A^2+A_op^2 +nilpotent}
Let $A$ be a two-sided skew brace. If $(A,\circ)$ is nilpotent of class $n$, then $(A^2+A_\op^2,+)$ is nilpotent of class at most $n+1$.
\end{lemma}
\begin{proof}
As $(A^2+A_\op^2)/(A^2\cap A_\op^2)$ is weakly trivial and $((A^2+A_\op^2)/(A^2\cap A_\op^2),\circ)$ is nilpotent of class at most $n$, also $((A^2+A_\op^2)/(A^2\cap A_\op^2),+)$ is nilpotent of class at most $n$ by \cref{cor: derived length weakly trivial}. By \cref{theorem: A two-sided additive group intersection A^2 Aop^2 abelian}, we have $A^2\cap A_\op^2\subseteq Z(A^2+A_\op^2,+)$, so we conclude that $(A^2+A_\op^2,+)$ is nilpotent of class at most $n+1$.
\end{proof}
\begin{theorem}
Let $A$ be a two-sided skew brace with nilpotent multiplicative group. Then its additive group is abelian-by-nilpotent and nilpotent-by-abelian.
\end{theorem}
\begin{proof}
As $(A/(A^2+A_\op^2),+)$ is abelian, the first claim follows by \cref{lem: circ nilpotent two-sided then A^2+A_op^2 +nilpotent}. For the second claim, recall from \cref{theorem: A two-sided additive group intersection A^2 Aop^2 abelian} that $(A^2\cap A_\op^2,+)$ is abelian. It follows from the assumption on $(A,\circ)$, together with \cref{prop: A/A^2 cap A^2op is weakly trivial} and \cref{cor: nilpotency group weakly trivial}, that $(A/(A^2\cap A_\op^2),+)$ is nilpotent.
\end{proof}
Recall the following definition 
\begin{definition}
A group satisfies the \emph{maximum condition on subgroups} if there exists no infinite strictly ascending chain of subgroups.
\end{definition}
It is well-known that the maximum condition on subgroups is preserved under taking subgroups, and forming quotients or extensions; see for example \cite[3.1.7]{Rob82} for a proof of the latter.

The following proposition is a reformulation of the main result of \cite{Wat68}. 
\begin{proposition}\label{prop: watters}
Let $A$ be a two-sided brace, then the following are equivalent:
\begin{enumerate}
    \item $(A,+)$ satisfies the maximum condition on subgroups,
    \item $(A,\circ)$ satisfies the maximum condition on subgroups.
\end{enumerate}
In this case, $A$ is strongly nilpotent and $(A,\circ)$ is nilpotent.
\end{proposition}
We now generalize the first part of the previous proposition to two-sided skew braces. A generalisation of the second part is given at the end of this section in \cref{theorem: generalisation watters nilpotent type}.
\begin{theorem}\label{theorem: additive iff multiplicative noetherian}
Let $A$ be a two-sided skew brace, then the following are equivalent:
\begin{enumerate}
    \item $(A,+)$ satisfies the maximum condition on subgroups,
    \item $(A,\circ)$ satisfies the maximum condition on subgroups.
\end{enumerate}
\end{theorem}
\begin{proof}
We first prove the statement for weakly trivial skew braces. Let $A$ be a weakly trivial skew brace such that $(A,+)$ satisfies the maximum condition on subgroups. Then so do $(A/A^2,+)$ and $(A/A^2_\op,+)$, hence also $(A/A^2,\circ)$ and $(A/A^2_\op,\circ)$. As $A$ embeds into $A/A^2\times A/A^2_\op$, we find that $(A,\circ)$ satisfies the maximum condition on subgroups. The other implication is proved similarly.

Next, let $A$ be any two-sided skew brace such that $(A,+)$ satisfies the maximum condition on subgroups. Then both $(A/(A^2\cap A^2_\op),+)$ and $(A^2\cap A^2_\op,+)$ satisfy the maximum condition on subgroups. As $A/(A^2\cap A^2_\op)$ is weakly trivial and $A^2\cap A^2_\op$ is a two-sided brace, we find that $(A/(A^2\cap A^2_\op),\circ)$ and $(A^2\cap A^2_\op,\circ)$ satisfy the maximum condition on subgroups. It follows that $(A,\circ)$ satisfies the maximum condition on subgroups. The other implication is proved similarly.
\end{proof}
It is natural to ask if \cref{theorem: additive iff multiplicative noetherian} can be generalized in the following way.
\begin{question}
Let $A$ be a two-sided skew brace that satisfies the maximum condition on subskew braces. Does this imply that the equivalent conditions of \cref{theorem: additive iff multiplicative noetherian} are satisfied?
\end{question}
For weakly trivial skew braces, this question can easily be answered affirmatively. The question therefore remains whether the same is true for all two-sided braces.

Similarly, the question arises whether we can replace the maximum condition on subgroups by the assumption that they are finitely generated. It is generally not true that if $(A,+)$ is finitely generated then $(A,\circ)$ is finitely generated, as the following example shows. 
\begin{example}
Consider the wreath product $G= C_2 \wr \Z$, also known as the lamplighter group, which is finitely generated but does not satisfy the maximum condition on subgroups as it contains the group $H=\bigoplus_{i\in \Z} C_2$. The exact factorization by $\Z$ and $H$ yields a skew brace with multiplicative group $H\times \Z$ and additive group $G$, see \cite[Theorem 3.3]{SV18}. As $H\times \Z$ is abelian, the obtained skew brace is clearly two-sided. However, $G$ is finitely generated while $H\times \Z$ is not.
\end{example}
On the other hand, it is not even known whether there exist two-sided braces with a finitely generated multiplicative group and non-finitely generated additive group. See \cite[Question 4.1]{Sys12} for a short discussion of an equivalent question.
\begin{theorem}\label{theorem: byott}
Let $A$ be a two-sided skew brace that satisfies the equivalent conditions of \cref{theorem: additive iff multiplicative noetherian}, then $(A,+)$ is soluble if and only if $(A,\circ)$ is soluble. In this case, $A$ is a soluble skew brace.
\end{theorem}
\begin{proof}
Similar to the proof of \cref{lem: circ nilpotent two-sided then A^2+A_op^2 +nilpotent} it is sufficient to prove the statement for weakly trivial skew braces and two-sided braces. For weakly trivial skew braces this was proved in \cref{cor: derived length weakly trivial}. For two-sided braces, the additive group is of course soluble and it follows from \cref{prop: watters} that both the brace and the multiplicative group are soluble.
\end{proof}
\begin{remark}
    Recall that one possible formulation of Byott's conjecture is that all finite skew braces of soluble type have a soluble multiplicative group \cite{Byo15}. \cref{theorem: byott} positively answers Byott's conjecture for all, not necessarily finite, two-sided skew braces satisfying the properties of \cref{theorem: additive iff multiplicative noetherian}.  We can not drop the condition of \cref{theorem: byott}, as there exist two-sided braces whose additive group is not finitely generated and whose multiplicative group is non-soluble; see \cite[Example 3.2]{Nas19}.
\end{remark}
Our next aim is to prove \cref{theorem: two-sided strongly nilpotent iff right iff left nilpotent} which states that, as is the case for two-sided braces, there is no distinction between left and right nilpotency of two-sided skew braces of nilpotent type.
\begin{lemma}\label{lem: distr center}
Let $A$ be a two-sided skew brace, then for all $a,b\in A$ and $c\in Z(A,+)$, we have $a*c,c*a\in Z(A,+)$ and
\begin{align*}
    c*(a+b)&=c*a+c*b,\\
    (a+b)*c&=a*c+b*c.
\end{align*}
\end{lemma}
\begin{proof}
As $Z(A,+)$ is characteristic in $(A,+)$, it is an ideal of $A$ by \cref{cor: char is ideal}. Hence, for $a\in A$ and $c\in Z(A,+)$ we find $c*a=-c+\lambda_a^\op (\overline{a}\circ c\circ a)\in Z(A,+)$ and $a*c=\lambda_a(c)-c\in Z(A,+)$. The second part of the statement now follows from \cref{lem: distributivity skew}.
\end{proof}
\begin{lemma}\label{lem: associativity *}
Let $A$ be a two-sided skew brace, then for all $a,b\in A$ and $c\in Z(A,+)$,
\begin{equation}
    (a*b)*c=a*(b*c).
\end{equation}
\end{lemma}
\begin{proof}
Let $a,b\in A$, $c\in Z(A,+)$. Using \cref{lem: distr center} and \cref{lem: distributivity skew} we find
\begin{align*}
    (a*b)*c&=(-a+a\circ b-b)*c\\
    &=-a*c+(a\circ b)* c-b*c\\
    &=-a*c+a*(b*c)+b*c+a*c-b*c\\
    &=a*(b*c).\qedhere
\end{align*}
\end{proof}
\begin{lemma}\label{lem: associativity * sets}
Let $A$ be a two-sided skew brace and $X,Y,Z$ subsets of $A$ with $X\subseteq Z(A,+)$. Then $(X*Y)*Z=X*(Y*Z)$.
\end{lemma}
\begin{proof}
By definition, $Y*Z$ consists of all $\sum_{i=1}^n\epsilon_i (b_i*c_i)$ with $n\in \N$, $\epsilon_i\in \{-1,1\}$, $b_i\in Y$, $c_i\in Z$. Therefore, $X*(Y*Z)$ is the additive subgroup generated by 
$$\left\{a*\left(\sum_{i=1}^n\epsilon_i (b_i*c_i)\right)\mid n\in \N, \epsilon_i\in \{-1,1\}, a\in X, b_i\in Y, c_i\in Z\right\}.$$
Using \cref{lem: distr center} we find that
$ a*(\sum_{i=1}^n\epsilon_i (b_i*c_i))=\sum_{i=1}^n\epsilon_i (a*(b_i*c_i))
$,
so $X*(Y*Z)$ is the additive subgroup generated by
\begin{equation}
    \{(a*(b*c)\mid a\in X, b\in Y, c\in Z\}.\label{eq: generators 1}
\end{equation}
A similar argument shows that $(X*Y)*Z$ is the additive subgroup generated by 
\begin{equation}
    \{(a*b)*c\mid a\in X, b\in Y, c\in Z\}.\label{eq: generators 2}
\end{equation}
It follows from \cref{lem: associativity *} that \eqref{eq: generators 1} and \eqref{eq: generators 2} are the same set.
\end{proof}
\begin{lemma}\label{lem: well behaving powers}
Let $A$ be a two-sided skew brace and $m\geq 1$ such that $A^{(m)}\subseteq Z(A,+)$, then for all $k\geq 1$,
\begin{align}
    A^{(m+k)}&=A^{(m)}*A^{(k)},\label{eq: intermediate}\\
    A^{(mk)}&=(A^{(m)})^{(k)}.\label{eq: final}
\end{align}
\end{lemma}
\begin{proof}
We first prove \eqref{eq: intermediate} by induction on $k$. For $k=1$ this is true by definition. For $k>1$ we use the induction hypothesis and \cref{lem: associativity * sets} to find
\begin{equation*}
    A^{(m+k)}=A^{(m+k-1)}*A=(A^{(m)}*A^{(k-1)})*A=A^{(m)}*(A^{(k-1)}*A)=A^{(m)}*A^{(k)}.
\end{equation*}

Next we prove \eqref{eq: final} by induction on $k$. For $k=1$ this is trivial. We can use the induction hypothesis and \eqref{eq: intermediate} to find that also for $k>1$,
\begin{equation*}
    A^{(mk)}=A^{(m(k-1)+m)}=A^{(m(k-1))}*A^{(m)}=(A^{(m)})^{(k-1)}*A^{(m)}=(A^{(m)})^{(k)}.\qedhere
\end{equation*}

\end{proof}
\begin{theorem}\label{theorem: two-sided strongly nilpotent iff right iff left nilpotent}
Let $A$ be a two-sided skew brace of nilpotent type, then the following properties are equivalent
\begin{enumerate}
    \item $A$ is left nilpotent,
    \item $A$ is right nilpotent,
    \item $A$ is strongly nilpotent.
\end{enumerate}
\end{theorem}
\begin{proof}
Because of \cref{theorem: left right strong nilpotency}, only the equivalence of 1 and 2 has to be proved. We prove the implication from 1 to 2 through induction on the nilpotency class $n$ of $(A,+)$. If $n=1$ then $A$ is a two-sided brace, so in particular left and right nilpotency coincide. Now assume that the claim is true for $n-1$ and let $A$ be a two-sided skew brace such that $(A,+)$ has nilpotency class $n$. Then $A/Z(A,+)$ is still left nilpotent and its additive group has nilpotency class $n-1$, so it is right nilpotent. Let $m$ be such that $A^{(m)}\subseteq Z(A,+)$. From the assumption that $A$ is left nilpotent, we find that in particular the subbrace $Z(A,+)$ is left, so also right, nilpotent. Let $k$ be such that $Z(A,+)^{(k)}=\{0\}$. By \cref{lem: well behaving powers} it follows that $A^{(mk)}=(A^{(m)})^{(k)}\subseteq Z(A,+)^{(k)}=\{0\}$, so $A$ is right nilpotent.

A similar argument proves the implication from 2 to 1.
\end{proof}
\begin{theorem}\label{theorem: generalisation watters nilpotent type}
Let $A$ be a two-sided skew brace that satisfies the equivalent conditions of \cref{theorem: additive iff multiplicative noetherian}, if $(A,+)$ is nilpotent then $(A,\circ)$ is nilpotent and $A$ is a strongly nilpotent skew brace.
\end{theorem}
\begin{proof}
We prove by induction on the nilpotency class $n$ of $(A,+)$ that $A$ is right nilpotent. Then by \cref{theorem: two-sided strongly nilpotent iff right iff left nilpotent} $A$ is strongly nilpotent and by \cite[Proposition 2.12]{JVAV22} $(A,\circ)$ is nilpotent. For $n=1$, the statement follows directly from \cref{prop: watters}. For $n>1$, we know that $(A/Z(A,+),+)$ has nilpotency class $n-1$, so by the induction hypothesis there exists some $m$ such that $A^{(m)}\subseteq Z(A,+)$. From the case $n=1$, we know that $Z(A,+)$ is right nilpotent hence $Z(A,+)^{(k)}=\{0\}$ for an appropriate choice of $k$. It follows by \cref{lem: well behaving powers} that $A^{(mk)}=(A^{(m)})^{(k)}\subseteq Z(A,+)^{(k)}=\{0\}$, so $A$ is right nilpotent.
\end{proof}

\section{Prime and semiprime two-sided skew braces}\label{section: prime and semiprime two-sided skew braces}
In \cite{KSV21} and its appendix, the following notions are introduced. 
\begin{definition}
Let $A$ be a skew brace. 
\begin{itemize}
\item $A$ is \emph{prime} if $I*J\neq \{0\}$ for any non-zero ideals $I$ and $J$.
\item $A$ is \emph{strongly prime} if every $*$-product of any number of non-zero ideals is non-zero.
\item $A$ is \emph{semiprime} if $I*I\neq \{0\}$ for any non-zero ideal $I$.
\item $A$ is \emph{strongly semiprime} if every $*$-product of any number of copies of a non-zero ideal $I$ is non-zero.
\end{itemize}
\end{definition}

For two-sided braces, both variations of (semi)primeness correspond with the usual notions for rings. As the authors note in \cite{KSV21}, it is an open question whether every prime, respectively semiprime, skew brace is a strongly prime, respectively strongly semiprime, skew brace. In this section we affirmatively answer this question for two-sided skew braces.
\begin{lemma}\label{lem: XY normal in circ then X*Y normal in circ}
Let $A$ be a two-sided skew brace and $X,Y$ subsets of $A$ which are normal in $(A,\circ)$, then $X*Y$ is normal in $(A,\circ)$.
\end{lemma}
\begin{proof}
Using \cref{prop: twosided iff mult conj additive automorphism} we find for arbitrary $a\in A$, $x\in X$, $y\in Y$,
\begin{align*}
    a\circ (x*y)\circ \overline{a}
    &=(a\circ x\circ \overline{a})*(a\circ y\circ \overline{a}),
\end{align*}
from which we conclude that $\{x*y\mid x\in X, y\in Y\}$ is a normal subset of $(A,\circ)$.
Once again using \cref{prop: twosided iff mult conj additive automorphism} we find that $X*Y$, which by definition is the additive subgroup generated by this set, is normal in $(A,\circ)$.
\end{proof}

\begin{lemma}\label{lem: X norm in circ I left ideal X*I left ideal}
Let $A$ be a skew left brace, $J$ a left ideal of $A$ and $X$ a normal subset of $(A,\circ)$. Then $X*J$ is a left ideal of $A$.
\end{lemma}
\begin{proof}
This follows from the fact that for all $a,b,c\in A$ we have the equality
\begin{equation*}
    \lambda_a(b*c)=\lambda_a(\lambda_b(c)-c)
    =\lambda_{a\circ b\circ \overline{a}}\lambda_a(c)-\lambda_a(c)
    =(a\circ b\circ \overline{a})*\lambda_a(c).\qedhere
\end{equation*}
\end{proof}

\begin{lemma}\label{lem: IJ is ideal under conditions}
Let $A$ be a two-sided skew brace, and $I$ and $J$ ideals. If $J\cap A_\op^2=\{0\}$ or $J\subseteq A_\op$, then $I*J$ is an ideal of $A$.
\end{lemma}
\begin{proof}
It follows from \cref{lem: XY normal in circ then X*Y normal in circ} and \cref{lem: X norm in circ I left ideal X*I left ideal} that $I*J$ is a left ideal which is moreover normal in $(A,\circ)$. Therefore it remains to show that $I*J$ is normal in $(A,+)$. We treat the two cases separately.

Assume $J\cap A_\op^2=\{0\}$, so in particular $I*_\op J=\{0\}$. Then $x\circ y=y+x$ or equivalently $x*y=-x+y+x-y$ for all $x\in I$, $y\in J$. We find that $I*J$ is in fact the commutator of the subgroups $(I,+)$ and $(J,+)$. As these are normal subgroups of $(A,+)$, also their commutator is normal in $(A,+)$.

Next, assume that $J\subseteq A_\op^2$ instead. As $I*J$ is the additive group generated by the elements $x*y$, where $x\in I$, $y\in J$, it is sufficient to prove that $-a+x*y+a\in I*J$ for all $a\in A$. From \cref{lem: distributivity skew} we find
\begin{equation*}
    0=x*(a-a)=x*a+a+x*(-a)-a,
\end{equation*}
hence $x*(-a)=-a-x*a+a$. From \cref{lem: twosided skew brace additive commutator A*A A*opA is zero} we find that $y,x*y$, which are elements of $A^2_\op$, commute with $x*a$, which is contained in $A^2$. Using this, in combination with our prior observation, we indeed find
\begin{align*}
    x*(-a+y+a)&=x*(-a)-a+x*(y+a)+a\\
    &=-a-x*a+x*y+y+x*a-y+a\\
    &=-a+x*y+a.\qedhere
\end{align*}
\end{proof}
\begin{theorem}
    Let $A$ be a two-sided skew brace, then $A$ is semiprime if and only if it is strongly semiprime.
\end{theorem}
\begin{proof}
It suffices to show the implication from left to right, we do so by contraposition. Assume that $A$ contains a non-zero ideal $I$ such that there exists a $*$-product of $n$ copies of $I$ which is zero. From \cite[Lemma 6.11]{KSV21} it follows that there exists some $n\geq 2$ such that $I_n=\{0\}$ where $I_1=I$ and $I_{k+1}=I_k*I_k$. 

Now if $I\cap A_\op^2=\{0\}$, then \cref{lem: IJ is ideal under conditions} implies that all $I_k$ are ideals, so in particular there exists some $k$ such that the ideal $I_k$ is non-zero and $I_k*I_k=\{0\}$. It follows that $A$ is not semiprime.

Next, assume that $I'=I\cap A_\op^2$ is non-trivial. Then $I'_n\subseteq I_n=\{0\}$, but as all the $I'_k$ are ideals by \cref{lem: IJ is ideal under conditions} we once again conclude that $A$ is not semiprime.
\end{proof}
\begin{theorem}
    Let $A$ be a two-sided skew brace, then $A$ is prime if and only if it is strongly prime.
\end{theorem}
\begin{proof}
We can restrict to proving the implication from left to right. Assume that $A$ is prime.
If there exists a non-zero ideal $I$ such that $I\cap A^2_\op=\{0\}$, then $I*A^2_\op\subseteq I\cap A^2_\op=\{0\}$. So either $A^2_\op=\{0\}$ or all ideals intersect $A^2_\op$ non-trivially. Assume that there exists a $*$-product $P$ of non-zero ideals of $A$ which is zero. 

If $A^2_\op=\{0\}$, so $A$ is almost trivial, then a $*$-product of two ideals is once again an ideal; this is easily seen directly or also follows from \cref{lem: IJ is ideal under conditions}. In particular, at some point in $P$ the $*$-product of two non-zero ideals gives zero and therefore $A$ is not prime.

In the case that all ideals intersect $A^2_\op$ non-trivially, we might replace all the ideals appearing in $P$ by their intersection with $A_\op^2$ to obtain a new product $P'$ which is also zero. But from \cref{lem: IJ is ideal under conditions} we find that every $*$-product in $P'$ gives an ideal and therefore we can find a $*$-product of two non-zero ideals in $P'$ which is zero.
\end{proof}
\section{Acknowledgments}
The author was supported by Fonds Wetenschappelijk Onderzoek -- Vlaanderen, via grant 1160522N
\bibliographystyle{amsalpha}
\bibliography{bib}
\end{document}